\documentclass[final,1p,times]{my_elsarticle}

\usepackage{amssymb}
\usepackage{amsmath}
\usepackage{amsthm}

\usepackage{a4wide}

\newtheorem{theorem}{Theorem}

\newtheorem{lemma}{Lemma}

\theoremstyle{remark}
\newtheorem{remark}{Remark}

\theoremstyle{definition}
\newtheorem{definition}{Definition}

\begin{document}

\begin{frontmatter}

\title{Higher-order infinite horizon variational problems\\
in discrete quantum calculus}

\author{Nat\'{a}lia Martins}
\ead{natalia@ua.pt}

\author{Delfim F. M. Torres}
\ead{delfim@ua.pt}

\address{Center for Research and Development in Mathematics and Applications\\
Department of Mathematics, University of Aveiro, 3810-193 Aveiro, Portugal}


\begin{abstract}
We obtain necessary optimality conditions
for higher-order infinite horizon problems
of the calculus of variations
via discrete quantum operators.
\end{abstract}


\begin{keyword}
Euler--Lagrange difference equations
\sep quantum calculus
\sep calculus of variations
\sep transversality conditions
\sep infinite horizon problems.

\smallskip

\MSC[2010] 39A13 \sep 49K05.
\end{keyword}

\end{frontmatter}


\section{Introduction}

Quantum difference operators are receiving an increase of interest
due to their applications in physics, economics and the calculus
of variations --- see
\cite{Almeida,withMiguel01,Cresson,malina:martins,MyID:187}
and references therein. Here we develop the quantum variational
calculus in the infinite horizon case.
Let $q >1$ and denote by $\mathcal{Q}$ the set
$\mathcal{Q}:=q^{\mathbb{N}_0}=\{q^n: n \in \mathbb{N}_0\}$.
In what follows $\sigma$ denotes the function defined by
$\sigma(t):=qt$ for all $t\in \mathcal{Q}$. For any
$k\in \mathbb{N}$, $\sigma^{k}:=\sigma\circ\sigma^{k-1}$,
where $\sigma^{0} = id$.
It is clear that $\sigma^{k}(t)=q^k t$.
For $f:\mathcal{Q}\rightarrow\mathbb{R}$
we define $f^{\sigma^k}:=f \circ \sigma^k$.
Fix $a \in \mathcal{Q}$ and $r \in \mathbb{N}$.
We are concerned with the following
higher-order $q$-variational problem:
\begin{equation}
\label{problem}
\begin{gathered}
\mathcal{J}(x(\cdot)) = \int_{a}^{+\infty}
L\left(t,(x\circ \sigma^r)(t), D_q[x\circ \sigma^{r-1}](t), \ldots,
D_q^{r-1}[x\circ \sigma](t), D^r_q[x](t)\right)d_q t  \longrightarrow \max \\
x(a)=\alpha_0, \quad D_q[x](a)=\alpha_1, \quad \ldots \quad D^{r-1}_q[x](a)=\alpha_{r-1},
\end{gathered}
\end{equation}
where $(u_1,\ldots, u_r, u_{r+1})\rightarrow L(t,u_1,\ldots,u_{r+1})$ is a
$C^1(\mathbb{R}^{r+1}, \mathbb{R})$ function for any $t \in \mathcal{Q}$,
and $\alpha_0$, \ldots, $\alpha_{r-1}$ are given real numbers.
The results of the paper are trivially generalized for the case of functions
$x:\mathcal{Q} \rightarrow\mathbb{R}^n$, $n \in \mathbb{N}$, but for simplicity of
presentation we restrict ourselves to the scalar case, \textrm{i.e.}, $n=1$.
In Section~\ref{Preliminary results} we present some
preliminary results and basic definitions.
Main results appear in Section~\ref{sec:mr}:
in \S\ref{Fundamental Lemmas} we prove some fundamental lemmas
of the calculus of variations for infinite horizon $q$-variational problems;
an Euler--Lagrange type equation and transversality conditions
for \eqref{problem} are obtained in \S\ref{E-L and Transversality}.


\section{Preliminaries}
\label{Preliminary results}

Let $f$ be a function defined on $\mathcal{Q}$.
By $D_q$ we denote the Jackson $q$-difference operator:
\begin{equation}
\label{eq:der:jac}
D_q[f](t):=\frac{f(qt)-f(t)}{(q-1)t}
\quad \forall t \in \mathcal{Q}.
\end{equation}
The higher-order $q$-derivatives are defined in the usual way:
the $r$th $q$-derivative, $r \in \mathbb{N}$,
of $f:\mathcal{Q}\rightarrow \mathbb{R}$ is the function
$D_{q}^{r}[f]: \mathcal{Q}\rightarrow \mathbb{R}$ given by
$D_{q}^{r}[f]:=D_{q}[D_{q}^{r-1}[f]]$,
where $D_{q}^{0}[f]:=f$.

The Jackson $q$-difference operator \eqref{eq:der:jac}
satisfies the following properties.

\begin{theorem}[\textrm{cf.} \cite{Kac}]
Let $f$ and $g$ be functions defined on $\mathcal{Q}$ and $t\in \mathcal{Q}$.
One has:
\begin{enumerate}
\item  $D_q[f] \equiv 0$ on $I$ if and only if $f$ is constant;

\item $D_q\left[  f+g\right]  \left(  t\right)  =D_q\left[
f\right]  \left(  t\right)  +D_q\left[  g\right]  \left(  t\right)$;

\item $D_q\left[  fg\right]  \left(  t\right)  =D_q\left[
f\right]  \left(  t\right)  g\left(  t\right)  +f\left(  qt\right)
D_q\left[  g\right]  \left(  t\right)$;

\item $\displaystyle D_q\left[  \frac{f}{g}\right]  \left(  t\right)
=\frac{D_q\left[  f\right]  \left(  t\right)  g\left(  t\right)
-f\left(  t\right)  D_q\left[  g\right]  \left(  t\right)  }{g\left(
t\right)  g\left(  qt\right)  }$ if $g\left(
t\right)  g\left(  qt\right)  \neq0$.
\end{enumerate}
\end{theorem}

Let $a \in \mathcal{Q}$ and $b:=aq^n\in \mathcal{Q}$
for some $n \in \mathbb{N}$. The $q$-integral of $f$
from $a$ to $b$ is defined by
$$
\int_{a}^{b}f(t) d_q t:= a(q-1) \sum_{k=0}^{n-1}q^{k}f(aq^{k}).
$$

\begin{theorem}[\textrm{cf.} \cite{Kac}]
If $a,b,c \in \mathcal{Q}$, $a
\leq c \leq b$, $\alpha, \beta \in \mathbb{R}$,
and $f,g:\mathcal{Q}\rightarrow \mathbb{R}$, then
\begin{enumerate}
\item $ \int_{a}^{b}\left(\alpha f(t) + \beta g(t) \right)
    d_q t= \alpha \int_{a}^{b}f(t)d_q t +
    \beta \int_{a}^{b}g(t)d_q t$;

\item $\int_{a}^{b}  f(t)d_q t =
- \int_{b}^{a} f(t)d_q t$;

\item $ \int_{a}^{a}  f(t)d_q t=0$;

\item $ \int_{a}^{b}  f(t)d_q t =
    \int_{a}^{c}  f(t)d_q t + \int_{c}^{b} f(t)d_q t$;

\item If $f(t)> 0$ for all $a \leq  t< b$, then $
    \ \int_{a}^{b}  f(t)d_q t > 0$;

\item $\int_{a}^{b}f(t)D_q[g](t)d_q t
=\left[f(t)g(t)\right]_{t=a}^{t=b}
-\int_{a}^{b}D_q[f](t)g(qt)d_q t$
\ \ ($q$-integration by parts formula);

\item $ \int_{a}^{b}D_q[f](t) d_q t= f(b)-f(a)$
\ \ (fundamental theorem of $q$-calculus);

\item $D_q\left[\int_{a}^{t}f(\tau)d_q \tau\right](t)=f(t)$.
\end{enumerate}
\end{theorem}

As usual, we define
$$
\int_{a}^{+\infty} f(t) d_q t := \lim_{b\rightarrow+\infty}
\int_a^b f(t) d_q t
$$
provided this limits exists (in $\overline{\mathbb{R}}
:=\mathbb{R}\cup\{-\infty, +\infty\}$).
We say that the improper $q$-integral converges
if this limit is finite; otherwise, we say that the
improper $q$-integral diverges.

In what follows all intervals are $q$-intervals, that is,
for $a,b \in \mathcal{Q}$,
$[a,b] := \{ t \in \mathcal{Q} : a \le t \le b\}$ and
$[a, +\infty[ := \{ t \in \mathcal{Q} : a \le t < +\infty \}$.

\begin{definition}
We say that $x:[a,+\infty[\rightarrow\mathbb{R}$
is an admissible path for problem \eqref{problem} if
$x(a)=\alpha_0,  D_q[x](a)=\alpha_1,
\ldots,  D^{r-1}_q[x](a)=\alpha_{r-1}$.
\end{definition}

There are several definitions of optimality
for problems with unbounded domain (see,
\textrm{e.g.}, \cite{Brock,Gale,SS,Weiz}).
Here we follow Brock's notion of optimality.

\begin{definition}
\label{def:weakMax}
Suppose that $a, T, T^\prime\in \mathcal{Q}$
are such that $T^\prime \geq T > a$.
We say that $x_{\ast}$ is weakly maximal to problem
\eqref{problem} if and only if $x_{\ast}$ is an admissible path and
\begin{multline*}
\lim_{T\rightarrow+\infty}\inf_{T^\prime \geq
T}\int_{a}^{T^\prime}\Biggl[L(t,(x\circ \sigma^r)(t),
D_q[x\circ \sigma^{r-1}](t), \ldots,
D_q^{r-1}[x\circ \sigma](t), D^r_q[x](t))\\
- L(t,(x_\ast\circ \sigma^r)(t), D_q[x_\ast\circ \sigma^{r-1}](t), \ldots,
D_q^{r-1}[x_\ast\circ \sigma](t), D^r_q[x_\ast](t))\Biggr]d_q t \le 0
\end{multline*}
for all admissible $x$.
\end{definition}

Note that in the case where the functional $\mathcal{J}$
of problem \eqref{problem} converges for all admissible paths,
the weak maximal path is optimal in the sense of the usual
definition of optimality. However, if every admissible
function $x$ yields an infinite value to the functional,
using the usual definition of optimality
each admissible path is an optimal path, showing that the standard
definition is not appropriate for problems with an unbounded domain.

Lemmas~\ref{lemmaderivadacomposta} and \ref{lemma_funcoes_admissiveis_1}
are an immediate consequence of the definition
of Jackson $q$-difference operator.

\begin{lemma}
\label{lemmaderivadacomposta}
For any $f:\mathcal{Q}\rightarrow \mathbb{R}$ and $t\in\mathcal{Q}$,
$D_q[f](\sigma(t))=\frac{1}{q} D_q[f\circ \sigma](t)$.
\end{lemma}

\begin{lemma}
\label{lemma_funcoes_admissiveis_1}
Assume $\eta:[a,+\infty[ \rightarrow \mathbb{R}$ is such that
$D_q^{i}[\eta](a)=0$ for all $i=0,1,\ldots, r$.
Then, $D_q^{i-1}[\eta\circ \sigma](a)=0$
for each $i=1,\ldots, r$.
\end{lemma}

The following basic result will be useful
in the proof of our main result
(Theorem~\ref{main result}).

\begin{theorem}[\textrm{cf.} \cite{Serge:Lang}]
\label{Serge:Lang}
Let $S$ and $T$ be subsets of a  normed vector space. Let $f$ be a
map defined on $T \times S$, having values in some complete normed
vector space. Let $v$ be adherent to $S$ and $w$ adherent to $T$.
Assume:
\begin{enumerate}
\item $\lim_{x\rightarrow v} f(t,x)$ exists for each $t \in T$;

\item $\lim_{t\rightarrow w} f(t,x)$ exists uniformly for  $x \in S$.
\end{enumerate}
Then the limits
$\lim_{t\rightarrow w} \lim_{x\rightarrow v} f(t,x)$,
$\lim_{x\rightarrow v} \lim_{t\rightarrow w}f(t,x)$,
and $\lim_{(t,x)\rightarrow (w,v)} f(t,x)$
all exist and are equal.
\end{theorem}


\section{Main results}
\label{sec:mr}

Before proving our main result (Theorem~\ref{main result}),
we need several preliminaries results. Namely,
we prove in \S\ref{Fundamental Lemmas} a higher-order
$q$-integration by parts formula and
three higher-order fundamental lemmas
for the $q$-calculus of variations.


\subsection{Fundamental lemmas}
\label{Fundamental Lemmas}

In our results we use the standard convention
that $\sum_{k=1}^{j} \gamma(k) = 0$ whenever $j=0$.

\begin{lemma}[Higher-order $q$-integration by parts formula]
\label{integration-parts-higher-order}
Let $r \in \mathbb{N}$, $a,b \in \mathcal{Q}$,
$a<b$, $f, g:[a,\sigma^r(b)]\rightarrow\mathbb{R}$.
For each $i=1,2,\ldots,r$ we have
\begin{multline*}
\int_a^b f(t) D_q^i[g \circ \sigma^{r-i}](t)d_q t
= (-1)^i \int_a^b \left(\frac{1}{q}\right)^{\frac{i(i-1)}{2}}
D_q^{i}[f](t) g^{\sigma^{r}}(t)d_q t\\
+ \left[ f(t) D_q^{i-1}[g \circ \sigma^{r-i}](t)
+ \sum_{k=1}^{i-1}(-1)^k D_q^{k}[f](t) D_q^{i-1-k}[g \circ
\sigma^{r-i+k}](t) \cdot
\prod_{j=1}^{k}\left(\frac{1}{q}\right)^{i-j}\right]_{a}^{b}.
\end{multline*}
\end{lemma}

\begin{proof}
We prove the lemma by mathematical induction.
If $r=1$, the result is obviously true from the $q$-integration by parts formula.
Assuming that the result holds for degree $r>1$, we will prove it for $r+1$.
Fix some $i=1,2,\ldots, r$. By the induction hypotheses, we get
\begin{equation*}
\begin{split}
\int_a^b &f(t) D_q^i[g \circ \sigma^{r+1-i}](t)d_q t
= \int_a^b f(t) D_q^i[g^\sigma \circ \sigma^{r-i}](t)d_q t\\
&=\left[f(t) D_q^{i-1}[g^\sigma \circ \sigma^{r-i}](t)
+ \sum_{k=1}^{i-1}(-1)^k D_q^{k}[f](t)  D_q^{i-1-k}[g^\sigma \circ \sigma^{r-i+k}](t)
\cdot \prod_{j=1}^{k}\left(\frac{1}{q}\right)^{i-j}\right]_{a}^{b}\\
&\qquad + (-1)^i \int_a^b \left(\frac{1}{q}\right)^{\frac{i(i-1)}{2}}
D_q^{i}[f](t)(g^\sigma)^{\sigma^{r}}(t)d_q t\\
&= \left[f(t) D_q^{i-1}[g \circ \sigma^{r+1-i}](t)
+ \sum_{k=1}^{i-1}(-1)^k D_q^{k}[f](t)
D_q^{i-1-k}[g \circ \sigma^{r+1-i+k}](t)
\cdot \prod_{j=1}^{k}\left(\frac{1}{q}\right)^{i-j}\right]_{a}^{b}\\
&\qquad + (-1)^i \int_a^b \left(\frac{1}{q}\right)^{\frac{i(i-1)}{2}}
D_q^{i}[f](t) g^{\sigma^{r+1}}(t)d_q t.
\end{split}
\end{equation*}
It remains to prove that the result is true for $i=r+1$.
Note that
$$
\int_a^b f(t) D_q^{r+1}[g](t)d_q t
= \int_a^b f(t)  D_q^{r}[D_q[g]](t)d_q t
$$
and, by the induction hypotheses for degree $r$ and $i=r$,
\begin{multline*}
\int_a^b f(t) D_q^{r+1}[g](t)d_q t
= (-1)^r \int_a^b \left(\frac{1}{q}\right)^{\frac{r(r-1)}{2}}
D_q^{r}[f](t) D_q[g](\sigma^r(t))d_q t\\
+ \left[f(t) D_q^{r-1}[D_q[g]](t)
+ \sum_{k=1}^{r-1}(-1)^k D_q^{k}[f](t)
D_q^{r-1-k}[D_q[g]\circ \sigma^k](t) \cdot
\prod_{j=1}^{k}\left(\frac{1}{q}\right)^{r-j}\right]_{a}^{b}.
\end{multline*}
From Lemma~\ref{lemmaderivadacomposta} we can write that
\begin{multline*}
\int_a^b f(t) D_q^{r+1}[g](t)d_q t
= \left[f(t) D_q^{r}[g](t) + \sum_{k=1}^{r-1}(-1)^k D_q^{k}[f](t)
D_q^{r-k}[g \circ \sigma^k](t)\cdot \left( \frac{1}{q}\right)^{k}
\prod_{j=1}^{k}\left(\frac{1}{q}\right)^{r-j}\right]_{a}^{b}\\
+ (-1)^r \int_a^b \left(\frac{1}{q}\right)^{\frac{r(r-1)}{2}}
\left( \frac{1}{q}\right)^{r} D_q^{r}[f](t)
D_q[g \circ \sigma^r](t)d_q t
\end{multline*}
and, by the $q$-integration by parts formula,
\begin{multline*}
\int_a^b f(t) D_q^{r+1}[g](t)d_q t
= \left[f(t) D_q^{r}[g](t) + \sum_{k=1}^{r-1}(-1)^k D_q^{k}[f](t)
D_q^{r-k}[g \circ \sigma^k](t)\cdot
\prod_{j=1}^{k}\left(\frac{1}{q}\right)^{r+1-j}\right]_{a}^{b}\\
+ \left[(-1)^r D_q^{r}[f](t) g^{\sigma^{r}}(t)\left(
\frac{1}{q}\right)^{\frac{r(r+1)}{2}}\right]_{a}^{b}
- (-1)^{r} \int_a^b \left(\frac{1}{q}\right)^{\frac{r(r+1)}{2}} D_q^{r+1}[f](t)
g^{\sigma^{r+1}}(t)d_q t.
\end{multline*}
We conclude that
\begin{multline*}
\int_a^b f(t) D_q^{r+1}[g](t)d_q t
= \left[f(t) D_q^{r}[g](t) + \sum_{k=1}^{r}(-1)^k D_q^{k}[f](t)
D_q^{r-k}[g \circ \sigma^k](t)\cdot
\prod_{j=1}^{k}\left(\frac{1}{q}\right)^{r+1-j}\right]_{a}^{b}\\
+ (-1)^{r+1} \int_a^b \left(\frac{1}{q}\right)^{\frac{r(r+1)}{2}}
D_q^{r+1}[f](t)  g^{\sigma^{r+1}}(t)d_q t,
\end{multline*}
proving that the result is true for $i=r+1$.
\end{proof}

The following lemma follows easily
(by contradiction and the properties of the $q$-integral).

\begin{lemma}
\label{teorema tecnico}
Suppose that $a \in \mathcal{Q}$ and
$f:[a,+\infty[\rightarrow \mathbb{R}$
is a function such that $f \geq 0$. If
$$
\lim_{T\rightarrow+\infty}\inf_{T^\prime \geq T}
\int_{a}^{T^\prime}f(t)d_q t=0,
$$
then $f=0$ on $[a,+\infty[$.
\end{lemma}

We now present two first-order fundamental lemmas of the
$q$-calculus of variations for infinite horizon variational
problems.

\begin{lemma}
\label{lemma2}
Let $a \in \mathcal{Q}$ and $f:[a,+\infty[\rightarrow \mathbb{R}$. If
$$
\lim_{T\rightarrow+\infty}\inf_{T^\prime \geq
T} \int_{a}^{T^\prime}f(t)D_q[\eta](t)d_q t=0
\quad \mbox{ for all } \ \eta :[a,+\infty[\rightarrow \mathbb{R}
\ \ \mbox{ such \ that}\ \  \eta(a)=0,
$$
then $f(t)=c$ for all $t\in [a,+\infty[$, where $c \in \mathbb{R}$.
\end{lemma}

\begin{proof}
Fix $T, T^\prime \in \mathcal{Q}$ such that $T^\prime \geq T >a$.
Let $c$ be a constant defined by the condition
$$
\int_{a}^{T^\prime}\left(f(\tau)-c\right)d_q \tau=0
$$
and let
$$
\eta(t)=\int_{a}^{t}\left(f(\tau)-c\right)d_q \tau.
$$
Clearly,  $D_q[\eta](t)=f(t)-c$ and
$$
\eta(a)=\int_{a}^{a}\left(f(\tau)-c\right)d_q \tau=0
\quad \mbox{and} \quad
\eta(T^\prime)=\int_{a}^{T^\prime}\left(f(\tau)-c\right)d_q \tau=0.
$$
Observe that
$$
\int_{a}^{T^\prime}\left(f(t)-c\right)D_q[\eta](t)d_q t
= \int_{a}^{T^\prime}\left(f(t)-c\right)^2d_q t
$$
and
$$
\int_{a}^{T^\prime}\left(f(t)-c\right)D_q[\eta](t)d_q t
= \int_{a}^{T^\prime}f(t)D_q[\eta](t)d_q t- c
\int_{a}^{T^\prime}D_q[\eta](t)d_q t=\int_{a}^{T^\prime}f(t)D_q[\eta](t)d_q t.
$$
Hence,
$$
\lim_{T\rightarrow+\infty}\inf_{T^\prime \geq T}
\int_{a}^{T^\prime}f(t)D_q[\eta](t)d_q t
= \lim_{T\rightarrow+\infty}\inf_{T^\prime \geq T}
\int_{a}^{T^\prime}\left(f(t)-c\right)^2d_q t=0,
$$
which shows, by Lemma~\ref{teorema tecnico}, that
$f(t)-c=0$ for all $t \in [a,+\infty[$.
\end{proof}

\begin{lemma}
\label{lemma4}
Let $f,g:[a,+\infty[\rightarrow \mathbb{R}$. If
$$
\lim_{T\rightarrow+\infty}\inf_{T^\prime \geq T}
\int_{a}^{T^\prime}\left(f(t)\eta(q t)
+ g(t) D_q[\eta](t)\right)d_q t=0
$$
for all $\eta:[a,+\infty[\rightarrow \mathbb{R}$
such that $\eta(a)=0$, then
$D_q[g](t)=f(t)$ for all $t\in [a,+\infty[$.
\end{lemma}

\begin{proof}
Fix $T, T^\prime \in \mathcal{Q}$ such that
$T^\prime\geq T>a$ and
define $A(t)=\int_{a}^{t} f(\tau)d_q \tau$. Then
$D_q[A](t)=f(t)$ for all $t \in [a,+\infty[$  and
$$
\int_{a}^{T^\prime} A(t)D_q[\eta](t)d_q t
= \left[ A(t)\eta(t) \right]_{a}^{T^\prime}
- \int_{a}^{T^\prime}D_q[A](t) \eta(q t)d_q t
=A(T^\prime)\eta(T^\prime) - \int_{a}^{T^\prime} f(t)
\eta(q t)d_q t.
$$
Restricting $\eta$ to those such that
$\eta(T^\prime)=0$, we obtain
$$
\lim_{T\rightarrow+\infty}\inf_{T^\prime \geq
T} \int_{a}^{T^\prime} \left(f(t)\eta(q t) + g(t)
D_q[\eta](t)\right)d_q t =
\lim_{T\rightarrow+\infty}\inf_{T^\prime \geq
T} \int_{a}^{T^\prime} \left(-A(t) + g(t) \right)D_q[\eta](t)d_q t=0.
$$
By Lemma~\ref{lemma2} we may conclude that there exists
$c \in \mathbb{R}$ such that
$-A(t) + g(t)=c$ for all $t \in  [a,+\infty[$. Therefore,
$D_q[A](t)=D_q[g](t)$ for all $t \in [a,+\infty[$, proving the
desired result.
\end{proof}

\begin{lemma}[Higher-order fundamental lemma of the $q$-calculus of variations I]
\label{Fund. Lemma 1}
Let $f_0, f_1, \ldots, f_r : [a,+\infty[ \rightarrow \mathbb{R}$. If
$$
\lim_{T\rightarrow+\infty}\inf_{T^\prime \geq T}
\int_{a}^{T^\prime} \left(\sum_{i=0}^{r}f_i(t)
D_q^{i}[\eta \circ \sigma^{r-i}](t) \right)
d_q t=0
$$
for all $\eta: [a,+\infty[ \rightarrow \mathbb{R}$ such that
$\eta(a)=0$, $D_q[\eta](a)=0$, \ldots, $D_q^{r-1}[\eta](a)=0$, then
\begin{equation*}
\sum_{i=0}^{r} (-1)^i
\left(\frac{1}{q}\right)^{\frac{i(i-1)}{2}}D_q^{i}[f_i](t) =0
\quad \forall t \in [a,+\infty[.
\end{equation*}
\end{lemma}

\begin{proof}
We proceed by mathematical induction.
If $r=1$, the result is true by Lemma~\ref{lemma4}.
Assume that the result is true for some $r > 1$.
We prove that the result is also true for $r+1$.
Suppose that
$$
\lim_{T\rightarrow+\infty}\inf_{T^\prime \geq T}
\int_{a}^{T^\prime}\left(\sum_{i=0}^{r+1}f_i(t)
D_q^{i}[\eta \circ \sigma^{r+1-i}](t)\right)d_q t=0
$$
for all $\eta:[a, +\infty[\rightarrow\mathbb{R}$ such that
$\eta(a)=0$, $D_q[\eta](a)=0$, \ldots, $D_q^{r}[\eta](a)=0$.
We need to prove that
$$
\sum_{i=0}^{r+1} (-1)^i
\left(\frac{1}{q}\right)^{\frac{i(i-1)}{2}}D_q^{i}[f_i](t) =0
\quad \forall t \in [a,+\infty[.
$$
Note that
$$
\int_{a}^{T^\prime}
\left(\sum_{i=0}^{r+1}f_i(t) D_q^{i}[\eta \circ \sigma^{r+1-i}](t)
\right)d_q t = \int_{a}^{T^\prime}
\left(\sum_{i=0}^{r}f_i(t) D_q^{i}[\eta \circ \sigma^{r+1-i}](t)
\right)d_q t \ +  \int_{a}^{T^\prime}  f_{r+1}(t)
D_q[D_q^{r}[\eta]](t) d_q t.
$$
Using the $q$-integration by parts formula in the last integral,
we obtain that
$$
\int_{a}^{T^\prime}  f_{r+1}(t) D_q[D_q^{r}[\eta]](t)
d_q t = \left[f_{r+1}(t)D_q^{r}[\eta](t)
\right]^{T^\prime}_{a} - \int_{a}^{T^\prime}
D_q[f_{r+1}](t)D_q^{r}[\eta](qt)d_q t.
$$
Since $D_q^{r}[\eta](a)=0$ and we can restrict ourselves
to those $\eta$ such that $D_q^{r}[\eta](T^\prime)=0$, then
$$
\int_{a}^{T^\prime}  f_{r+1}(t) D_q[D_q^{r}[\eta]](t)
d_q t  = - \int_{a}^{T^\prime} D_q[f_{r+1}](t)D_q^{r}[\eta](\sigma(t))d_q t.
$$
By Lemma~\ref{lemmaderivadacomposta},
$$
\int_{a}^{T^\prime}  f_{r+1}(t) D_q[D_q^{r}[\eta]](t)
d_q t  = - \int_{a}^{T^\prime} D_q[f_{r+1}](t)\left(
\frac{1}{q}\right)^{r}D_q^{r}[\eta \circ \sigma](t)d_q t.
$$
Hence,
\begin{equation*}
\begin{split}
\int_{a}^{T^\prime} & \left(
\sum_{i=0}^{r+1}f_i(t)
D_q^{i}[\eta \circ \sigma^{r+1-i}](t) \right)d_q t \\
& = \int_{a}^{T^\prime}
\left(\sum_{i=0}^{r}f_i(t)
D_q^{i}[\eta \circ \sigma^{r+1-i}](t)\right)d_q t
- \int_{a}^{T^\prime} D_q[f_{r+1}](t)\left(
\frac{1}{q}\right)^{r}D_q^{r}[\eta \circ \sigma](t)d_q t\\
&= \int_{a}^{T^\prime} \left(\sum_{i=0}^{r-1}f_i(t)
D_q^{i}[\eta^\sigma \circ \sigma^{r-i}](t)
+ \left( f_r(t) - D_q[f_{r+1}](t)\left(
\frac{1}{q}\right)^{r}\right) D_q^{r}[\eta \circ \sigma](t)
\right)d_q t
\end{split}
\end{equation*}
and, therefore,
\begin{multline*}
\lim_{T\rightarrow+\infty}\inf_{T^\prime \geq
T} \int_{a}^{T^\prime} \left(
\sum_{i=0}^{r+1}f_i(t)
D_q^{i}[\eta \circ \sigma^{r+1-i}](t) \right)d_q t \\
=\lim_{T\rightarrow+\infty}\inf_{T^\prime \geq
T} \int_{a}^{T^\prime}
\left(\sum_{i=0}^{r-1}f_i(t) D_q^{i}[\eta^\sigma \circ \sigma^{r-i}](t)
+ \left( f_r(t)-D_q[f_{r+1}](t)\left(
\frac{1}{q}\right)^{r}\right) D_q^{r}[\eta \circ \sigma](t)
\right)d_q t=0.
\end{multline*}
By Lemma~\ref{lemma_funcoes_admissiveis_1},
$\eta^\sigma(a)=0$, $D_q[\eta \circ \sigma](a)= 0$,
\ldots, $D_q^{r-1}[\eta \circ \sigma](a)=0$.
Then, by the induction hypothesis, we conclude that
$$
\sum_{i=0}^{r-1} (-1)^i
\left(\frac{1}{q}\right)^{\frac{i(i-1)}{2}}D_q^{i}[f_i](t) \ +
\ (-1)^r \left(\frac{1}{q}\right)^{\frac{r(r-1)}{2}}
D_q^{r}\left[f_r - \left(\frac{1}{q}\right)^r D_q[f_{r+1}]\right](t)
=0 \quad \forall t \in [a,+\infty[ \, ,
$$
which is equivalent to
$\sum_{i=0}^{r+1} (-1)^i
\left(\frac{1}{q}\right)^{\frac{i(i-1)}{2}}D_q^{i}[f_i](t)
=0$ for all  $t \in [a,+\infty[$.
\end{proof}

\begin{lemma}[Higher-order fundamental lemma of the $q$-calculus of variations II]
\label{first-transversality}
Let $f_0, f_1, \ldots, f_r:[a,+\infty[ \rightarrow \mathbb{R}$. If
$$
\lim_{T\rightarrow+\infty}  \inf_{T^\prime \geq T}
\int_{a}^{T^\prime}  \left(\sum_{i=0}^{r}f_i
(t) D_q^{i}[\eta \circ \sigma^{r-i}](t) \right)d_q t=0
$$
for all $\eta: [a, +\infty[\rightarrow  \mathbb{R}$ such that
$\eta(a)=0$, $D_q[\eta](a)=0$, \ldots, $D_q^{r-1}[\eta](a)=0$, then
$$
\lim_{T\rightarrow+\infty} \inf_{T^\prime \geq T}
\left\{f_r(T^\prime)\cdot D_q^{r-1}[\eta](T^\prime)\right\} =0.
$$
\end{lemma}

\begin{proof}
Note that
\begin{equation*}
\begin{split}
\int_{a}^{T^\prime} & \left(\sum_{i=0}^{r}
f_i(t) D_q^{i}[\eta \circ \sigma^{r-i}](t) d_q t\right)
= \int_{a}^{T^\prime} f_0(t) \eta^{\sigma^r}(t)d_q t
+ \sum_{i=1}^{r}\left( \int_{a}^{T^\prime}
f_i(t)D_q^{i}[\eta \circ \sigma^{r-i}](t)d_q t\right)\\
&= \int_{a}^{T^\prime} f_0(t) \eta^{\sigma^r}(t)d_q t
+ \sum_{i=1}^{r} \left((-1)^i \int_{a}^{T^\prime}
\left(\frac{1}{q}\right)^{\frac{i(i-1)}{2}} D_q^{i}[f_i](t)
\eta^{\sigma^r}(t)d_q t \right)\\
&\quad +\sum_{i=1}^{r} \left[f_i(t)D_q^{i-1}[\eta \circ \sigma^{r-i}](t)
+\sum_{k=1}^{i-1}(-1)^k D_q^{k}[f_i](t) D_q^{i-1-k}[\eta \circ \sigma^{r-i+k}](t)
\cdot \prod_{j=1}^{k}\left(\frac{1}{q}\right)^{i-j}\right]_{a}^{T^\prime}\\
&= \int_{a}^{T^\prime} \left(f_0(t) + \sum_{i=1}^{r}
(-1)^i \left(\frac{1}{q}\right)^{\frac{i(i-1)}{2}} D_q^{i}[f_i](t)\right)
\cdot \eta^{\sigma^r}(t)d_q t \\
&\quad + \sum_{i=1}^{r} \left[ \left(f_i(t)D_q^{i-1}[\eta \circ \sigma^{r-i}](t)
+\sum_{k=1}^{i-1} (-1)^k D_q^{k}[f_i](t) D_q^{i-1-k}[\eta \circ \sigma^{r-i+k}](t)
\cdot \prod_{j=1}^{k}\left(\frac{1}{q}\right)^{i-j}\right)\right]_{a}^{T^\prime}\\
&= \int_{a}^{T^\prime} \left( \sum_{i=0}^{r} (-1)^i
\left(\frac{1}{q}\right)^{\frac{i(i-1)}{2}} D_q^{i}[f_i](t)\right)
\cdot \eta^{\sigma^r}(t)d_q t \\
& \quad + \sum_{i=1}^{r-1} \left[ \left(f_i(t)
D_q^{i-1}[\eta \circ \sigma^{r-i}](t)
+ \sum_{k=1}^{i-1} (-1)^k D_q^{k}[f_i](t)
D_q^{i-1-k}[\eta \circ \sigma^{r-i+k}](t)
\cdot \prod_{j=1}^{k}\left(\frac{1}{q}\right)^{i-j}
\right)\right]_{a}^{T^\prime}\\
& \quad + \left [ f_r(t)D_q^{r-1}[\eta](t) + \sum_{k=1}^{r-1}
(-1)^k D_q^{k}[f_r](t) D_q^{r-1-k}[\eta \circ \sigma^{k}](t)
\cdot \prod_{j=1}^{k}\left(\frac{1}{q}\right)^{r-j} \right ]_{a}^{T^\prime},
\end{split}
\end{equation*}
where in the second equality we use Lemma~\ref{integration-parts-higher-order}.
Applying now Lemma~\ref{Fund. Lemma 1} we get
\begin{equation*}
\begin{split}
\int_{a}^{T^\prime} &\left(\sum_{i=0}^{r}
f_i(t) D_q^{i}[\eta \circ \sigma^{r-i}](t) d_q t\right)\\
&= \sum_{i=1}^{r-1} \left[ \left(f_i(t)D_q^{i-1}[\eta \circ \sigma^{r-i}](t)
+\sum_{k=1}^{i-1} (-1)^k D_q^{k}[f_i](t) D_q^{i-1-k}[\eta
\circ \sigma^{r-i+k}](t) \cdot \prod_{j=1}^{k}\left(\frac{1}{q}\right)^{i-j}
\right)\right]_{a}^{T^\prime}\\
&\quad + \left[ f_r(t)D_q^{r-1}[\eta](t) +
\sum_{k=1}^{r-1} (-1)^k D_q^{k}[f_r](t) D_q^{r-1-k}[\eta
\circ \sigma^{k}](t) \cdot \prod_{j=1}^{k}\left(\frac{1}{q}\right)^{r-j}\right]_{a}^{T^\prime}.
\end{split}
\end{equation*}
Therefore, restricting the variations $\eta$ to those such that
$$
D_q^{k-1}[\eta \circ \sigma^{r-k}](T^\prime)
=D_q^{k-1}[\eta \circ \sigma^{r-k}](a)=0,
\quad \forall k=1,2,\ldots, r-1,
$$
$$
D_q^{r-1-k}[\eta \circ \sigma^{k}](T^\prime)
=D_q^{r-1-k}[\eta \circ \sigma^{k}](a)=0,
\quad \forall k=1,2,\ldots, r-1,
$$
we get
$$
\lim_{T\rightarrow+\infty}\inf_{T^\prime \geq
T} \int_{a}^{T^\prime}
\left(\sum_{i=0}^{r}f_i(t) D_q^{i}[\eta \circ \sigma^{r-i}](t)
\right)d_q t=0
\Rightarrow
\lim_{T\rightarrow+\infty}\inf_{T^\prime \geq
T} \left\{ f_r(T^\prime) D_q^{r-1}[\eta](T^\prime) \right\}=0
$$
proving the desired result.
\end{proof}

\begin{lemma}[Higher-order fundamental lemma of the $q$-calculus of variations III]
\label{Fund. Lemma 3}
Let $f_0, f_1, \ldots, f_r:[a,+\infty[ \rightarrow \mathbb{R}$. If
$$  \lim_{T\rightarrow+\infty}  \inf_{T^\prime \geq
T} \displaystyle \int_{a}^{T^\prime}  \left(\sum_{i=0}^{r}f_i
(t) D_q^{i}[\eta \circ \sigma^{r-i}](t) \right)d_q t=0$$
for all $\eta: [a, +\infty[\rightarrow \mathbb{R}$ such that
$\eta(a)=0$, $D_q[\eta](a)=0$, \ldots, $D_q^{r-1}[\eta](a)=0$, then
$$
\lim_{T\rightarrow+\infty} \inf_{T^\prime \geq T}
\left\{\left(f_{r-(k-1)} (T^\prime) + \sum_{i=1}^{k-1}
(-1)^{i} D_q^{i}[f_{r-(k-1)+i}](T^\prime)\cdot
\prod_{j=1}^{i}\left(\frac{1}{q}\right)^{r-(k-1)+(j-1)}\right)
\cdot D_q^{r-k}[\eta\circ \sigma^{k-1}](T^\prime)\right\} =0
$$
for $k=1,2,\ldots,r$.
\end{lemma}

\begin{proof}
We prove the lemma by mathematical induction.
For $r=1$, using the $q$-integration by parts formula
and Lemma~\ref{Fund. Lemma 1}, we obtain
$\lim_{T\rightarrow+\infty} \inf_{T^\prime \geq T}
f_1(T^\prime) \eta(T^\prime)=0$,
showing that the result is true for $r=1$.
Assuming that the result holds for degree $r>1$,
we will prove it for $r+1$.
Suppose that
$$
\lim_{T\rightarrow+\infty}\inf_{T^\prime \geq
T} \int_{a}^{T^\prime}
\left(\sum_{i=0}^{r+1}f_i(t) D_q^{i}[\eta \circ \sigma^{r+1-i}](t)
\right)d_q t=0
$$
for all $\eta:[a, +\infty[\rightarrow\mathbb{R}$ such that
$\eta(a)=0$, $D_q[\eta](a)=0$, \ldots, $D_q^{r}[\eta](a)=0$.
We need to prove
\begin{multline}
\label{tese_de_inducao}
\lim_{T\rightarrow+\infty} \inf_{T^\prime \geq T}
\Biggl\{\left (f_{r+1-(k-1)} (T^\prime)
+ \sum_{i=1}^{k-1} (-1)^{i} D_q^{i}[f_{r+1-(k-1)+i}](T^\prime)
\cdot \displaystyle\prod_{j=1}^{i}\left(\frac{1}{q}\right)^{r+1 -(k-1)+(j-1)} \right)\\
\cdot D_q^{r+1-k}[\eta \circ \sigma^{k-1}](T^\prime)\Biggr\} =0
\end{multline}
for $k=1,2,\ldots, r,r+1$.
Fix some $k=2,\ldots, r,r+1$. The main ideia of the proof is that the
$k$-transversality condition for the variational problem of order $r+1$
is obtained from the $k-1$ transversality condition for the variational problem of order $r$.
Using the same techniques as in Lemma~\ref{Fund. Lemma 1}, we prove that
\begin{multline*}
\lim_{T\rightarrow+\infty}\inf_{T^\prime \geq
T} \int_{a}^{T^\prime}
\left(\sum_{i=0}^{r+1}f_i(t) D_q^{i}[\eta \circ \sigma^{r+1-i}](t)
\right)d_q t=0\\
\Rightarrow \lim_{T\rightarrow+\infty}\inf_{T^\prime \geq
T} \left\{\int_{a}^{T^\prime}
\left(\sum_{i=0}^{r-1}f_i(t) D_q^{i}[\eta^\sigma \circ \sigma^{r-i}](t)
+ \left( f_r(t) - D_q[f_{r+1}](t)\left(
\frac{1}{q}\right)^{r}\right) D_q^{r}[\eta \circ \sigma](t)
\right)d_q t\right\}=0.
\end{multline*}
Since, by Lemma \ref{lemma_funcoes_admissiveis_1},
$\eta^\sigma(a)=0$, $D_q[\eta \circ \sigma](a)= 0$,
\ldots, $D_q^{r-1}[\eta \circ \sigma](a)=0$, then,
by the induction hypothesis for $k-1$, we conclude that
$$
\begin{array}{lcl}
& & \displaystyle \lim_{T\rightarrow+\infty} \inf_{T^\prime \geq T}
\Biggl\{\Biggl(f_{r-(k-2)} (T^\prime) + \sum_{i=1}^{k-3}
(-1)^{i} D_q^{i}[f_{r-(k-2)+i}](T^\prime)\cdot
\prod_{j=1}^{i}\Big(\frac{1}{q}\Big)^{r-(k-2)+(j-1)}\\
&&
\quad  \quad \quad \quad + (-1)^{k-2}D_q^{k-2}[f_r](T^\prime)
\cdot \displaystyle \prod_{j=1}^{k-2}\Big(\frac{1}{q}\Big)^{r-(k-2)+(j-1)}\\
&&
\quad  \quad \quad \quad  + (-1)^{k-1}D_q^{k-1}[f_{r+1}](T^\prime)
\cdot\displaystyle \prod_{j=1}^{k-2}\Big(\frac{1}{q}\Big)^{r-(k-2)
+(j-1)}\Big(\frac{1}{q}\Big)^{r}\Biggr)
\cdot D_q^{r-(k-1)}[\eta^\sigma\circ \sigma^{k-2}](T^\prime)\Biggr\} =0,
\end{array}
$$
which is equivalent to
$$
\lim_{T\rightarrow+\infty} \inf_{T^\prime \geq T}
\left\{\left(f_{r-(k-2)} (T^\prime)
+ \sum_{i=1}^{k-1} (-1)^{i} D_q^{i}[f_{r-(k-2)+i}](T^\prime)
\cdot \prod_{j=1}^{i}\left(\frac{1}{q}\right)^{r-(k-2)+(j-1)}\right)
\cdot D_q^{r-(k-1)}[\eta^{\sigma^{k-1}}](T^\prime)\right\} =0
$$
and proves equation \eqref{tese_de_inducao} for $k=2,3,\ldots, r, r+1$.
It remains to prove \eqref{tese_de_inducao} for $k=1$.
This condition follows from  Lemma~\ref{first-transversality}.
\end{proof}


\subsection{Euler--Lagrange equation and transversality conditions}
\label{E-L and Transversality}

We are now in conditions to prove a first-order
necessary optimality condition for the higher-order
infinite horizon $q$-variational problem.
In what follows $\partial_i L$ denotes the partial
derivative of $L$ with respect to its $i$th argument.
For simplicity of expressions, we introduce the operator
$\langle \cdot \rangle$ defined by
$$
\langle x\rangle(t) :=
\left(t,(x\circ \sigma^r)(t), D_q[x\circ \sigma^{r-1}](t),
\ldots, D_q^{r-1}[x\circ \sigma](t), D^r_q[x](t)\right).
$$

\begin{theorem}
\label{main result}
Suppose that the optimal path to problem \eqref{problem} exists and
is given by $x_{\ast}$. Let $\eta:[a,+\infty[ \rightarrow \mathbb{R}$ be such that
$\eta(a)=0, D_q[\eta](a)=0, \ldots, D_q^{r-1}[\eta](a)=0$. Define
\begin{equation*}
\begin{split}
A(\varepsilon, T^\prime)&:= \int_{a}^{T^\prime}
\frac{L\langle x_{\ast}+\epsilon \eta\rangle(t)
- L\langle x_{\ast} \rangle(t)}{\varepsilon} \, d_q t,\\
V(\varepsilon, T) &:= \inf_{T^\prime \geq
T}\int_{a}^{T^\prime} \left(
L\langle x_{\ast}+\epsilon \eta\rangle(t)
- L\langle x_{\ast} \rangle(t)\right) \, d_q t,\\
V(\varepsilon) &:= \lim_{T\rightarrow+\infty} V(\varepsilon, T).
\end{split}
\end{equation*}
Suppose that
\begin{enumerate}
\item $\lim_{\varepsilon \rightarrow 0}
\frac{V(\varepsilon, T) }{\varepsilon}$ exists for all $T$;

\item $\lim_{T\rightarrow+\infty}\frac{V(\varepsilon,
T)}{\varepsilon}$ exists uniformly for $\varepsilon$;

\item For every $T^\prime > a$, $T > a$,
and $\varepsilon\in \mathbb{R}\setminus\{0\}$,
there is a sequence
$\left(A(\varepsilon, T^\prime_n)\right)_{n \in \mathbb{N}}$
such that
$\lim_{n \rightarrow +\infty} A(\varepsilon, T^\prime_n)
= \inf_{T^\prime \geq T} A(\varepsilon, T^\prime)$
uniformly for $\varepsilon$.
\end{enumerate}
Then $x_{\ast}$ satisfies the Euler--Lagrange equation
\begin{equation}
\label{E-L equation}
\sum_{i=0}^{r} (-1)^i
\left(\frac{1}{q}\right)^{\frac{i(i-1)}{2}}
D_q^{i}\left[\partial_{i+2} L\langle x\rangle\right](t)=0
\end{equation}
for all $t \in [a,+\infty[$,
and the $r$ transversality conditions
\begin{equation}
\label{tranversality}
\displaystyle \lim_{T\rightarrow+\infty} \inf_{T^\prime \geq T}
\left\{\left( \partial_{r+2-(k-1)} L\langle x\rangle(T^\prime)
+ \sum_{i=1}^{k-1} (-1)^{i} D_q^{i}\left[\partial_{r+2-(k-1)+i}
L\langle x\rangle\right](T^\prime) \cdot \Psi_i\right)
\cdot D_q^{r-k}[x \circ \sigma^{k-1}](T^\prime)\right\} =0,
\end{equation}
$k=1,2,\ldots,r$, where
$\Psi_i= \prod_{j=1}^{i}\left(\frac{1}{q}\right)^{r-(k-1)+(j-1)}$.
\end{theorem}

\begin{proof}
Using the notion of weak maximality, if $x_{\ast}$ is optimal, then
$V(\varepsilon) \leq 0$ for every $\varepsilon \in \mathbb{R}$.
Since $V(0)=0$, then 0 is an extremal of $V$.
We prove that $V$ is differentiable at $t=0$,
hence $V^\prime(0)=0$. Note that
\begin{equation*}
\begin{split}
V^\prime(0)
&= \lim_{\varepsilon \rightarrow 0} \frac{V(\varepsilon)}{\varepsilon}
= \lim_{\varepsilon \rightarrow 0}
\lim_{T\rightarrow+\infty}\frac{V(\varepsilon, T) }{\varepsilon}\\
&= \lim_{T\rightarrow+\infty}
\lim_{\varepsilon \rightarrow 0} \frac{V(\varepsilon, T) }{\varepsilon}
\quad (\mbox{by hypothesis \emph{1} and \emph{2} and Theorem~\ref{Serge:Lang})}\\
&= \lim_{T\rightarrow+\infty}
\lim_{\varepsilon \rightarrow 0} \inf_{T^\prime \geq T} A(\varepsilon, T^\prime)\\
&= \lim_{T\rightarrow+\infty}
\lim_{\varepsilon \rightarrow 0} \lim_{n \rightarrow +\infty} A(\varepsilon, T^\prime_n)
\quad (\mbox{by hypothesis \emph{3}) }\\
&= \lim_{T\rightarrow+\infty}
\lim_{n \rightarrow +\infty} \lim_{\varepsilon \rightarrow 0} A(\varepsilon, T^\prime_n)
\quad (\mbox{by hypothesis \emph{3} and Theorem \ref{Serge:Lang}) }\\
&= \lim_{T\rightarrow+\infty}
\inf_{T^\prime \geq T}  \lim_{\varepsilon \rightarrow 0} A(\varepsilon, T^\prime)
\quad (\mbox{by hypothesis \emph{3})}\\
&= \lim_{T\rightarrow+\infty}
\inf_{T^\prime \geq T}  \lim_{\varepsilon \rightarrow 0}
\int_{a}^{T^\prime} \frac{L\langle x_{\ast}+\epsilon \eta\rangle(t)
- L\langle x_{\ast} \rangle(t)}{\varepsilon} \, d_q t\\
&= \lim_{T\rightarrow+\infty}  \inf_{T^\prime \geq T}
\int_{a}^{T^\prime} \lim_{\varepsilon \rightarrow 0}
\frac{L\langle x_{\ast}+\epsilon \eta\rangle(t)
- L\langle x_{\ast} \rangle(t)}{\varepsilon} \, d_q t\\
&= \lim_{T\rightarrow+\infty}  \inf_{T^\prime \geq T}
\int_{a}^{T^\prime}  \left(\sum_{i=0}^{r}\partial_{i+2}
L\langle x_{\ast} \rangle(t)
\cdot D_q^{i}[\eta \circ \sigma^{r-i}](t)\right) \, d_q t
\end{split}
\end{equation*}
and hence
$$
\lim_{T\rightarrow+\infty}  \inf_{T^\prime \geq
T} \displaystyle \int_{a}^{T^\prime}  \left(\sum_{i=0}^{r}\partial_{i+2}
L\langle x_{\ast} \rangle(t)
\cdot D_q^{i}[\eta \circ \sigma^{r-i}](t) \right) \, d_q t=0.
$$
Using Lemma~\ref{Fund. Lemma 1} we conclude that
\begin{equation*}
\sum_{i=0}^{r} (-1)^i
\left(\frac{1}{q}\right)^{\frac{i(i-1)}{2}}D_q^{i}\left[\partial_{i+2}
L\langle x_{\ast}\rangle\right](t)=0
\end{equation*}
for all $t \in [a,+\infty[$, proving that $x_{\ast}$ satisfy the
Euler--Lagrange equation \eqref{E-L equation}.
By Lemma~\ref{Fund. Lemma 3}, for $k=1,2,\ldots, r$,
\begin{equation}
\label{tranversality p}
\displaystyle \lim_{T\rightarrow+\infty} \inf_{T^\prime \geq T}
\left\{\left(\partial_{r+2-(k-1)} L\langle x_{\ast}\rangle(T^\prime)
+ \sum_{i=1}^{k-1} (-1)^{i} D_q^{i}\left[\partial_{r+2-(k-1)+i} L\langle
x_{\ast}\rangle\right](T^\prime) \cdot \Psi_i \right)
\cdot D_q^{r-k}[\eta\circ \sigma^{k-1}](T^\prime)\right\} =0,
\end{equation}
where $\Psi_i= \prod_{j=1}^{i}\left(\frac{1}{q}\right)^{r-(k-1)+(j-1)}$.
Consider $\eta$ defined by
$\eta(t)=\alpha(t) x_{\ast}(t)$, $t \in [a, +\infty[$,
where $\alpha: [a, +\infty[ \rightarrow \mathbb{R}$
satisfy $\alpha (a)=0, D_q[\alpha](a)=0,
\ldots, D_q^{r-1}[\alpha](a)=0$, and there exists
$T_0\in \mathcal{Q}$ such that $\alpha(t)=\beta \in
\mathbb{R}\setminus\{0\}$ for all $t> T_0$.
Note that
$\eta(a)=0, D_q[\eta](a)=0, \ldots, D_q^{r-1}[\eta](a)=0$.
Substituting $\eta$ in equation \eqref{tranversality p} we conclude that
$$
\lim_{T\rightarrow+\infty} \inf_{T^\prime \geq T}
\left\{\left(\partial_{r+2-(k-1)} L\langle x_{\ast}\rangle(T^\prime)
+ \sum_{i=1}^{k-1} (-1)^{i} D_q^{i}\left[\partial_{r+2-(k-1)+i}
L\langle x_{\ast}\rangle\right](T^\prime)
\cdot \Psi_i\right)\cdot D_q^{r-k}[x_{\ast}
\circ \sigma^{k-1}](T^\prime)\right\} =0,
$$
proving that $x_{\ast}$ satisfy the transversality
condition \eqref{tranversality} for all $k=1,2,\ldots,r$.
\end{proof}

\begin{remark}
For the simplest case $r=1$ we obtain from Theorem~\ref{main result}
the Euler--Lagrange equation
$$
D_q\left[s \rightarrow \partial _3 L\left(s, x(qs), D_q[x](s)\right)\right](t)
= \partial_2 L\left(t, x(qt), D_q[x](t)\right)
$$
and the transversality condition
$\lim_{T\rightarrow+\infty} \inf_{T^\prime \geq T} \left\{\partial_3
L\left(T^\prime, x(qT^\prime), D_q[x](T^\prime)\right)
\cdot x(T^\prime)\right\}=0$.
However, when $r > 1$, Theorem~\ref{main result}
gives more than one transversality condition.
Indeed, for an infinite horizon
variational problem of order $r$ one has
$r$ transversality conditions and, for each $k=1,2,\ldots, r$,
the $k$th transversality condition has exactly $k$ terms.
This improves the results of \cite{Nitta:et:all:2009}.
\end{remark}


\small


\section*{Acknowledgements}

Work supported by {\it FEDER} funds through
{\it COMPETE} (Operational Programme Factors of Competitiveness)
and by Portuguese funds through the
{\it Center for Research and Development
in Mathematics and Applications} (University of Aveiro)
and {\it FCT} (The Portuguese Foundation for Science and Technology),
within project PEst-C/MAT/UI4106/2011
with COMPETE number FCOMP-01-0124-FEDER-022690.



\end{document}